\documentclass[12pt]{amsart}

\usepackage{amscd,amssymb,amsopn,amsmath,amsthm,mathrsfs,graphics,amsfonts,enumerate,verbatim,calc
}

\usepackage[all]{xy}

\usepackage{color}

\usepackage[OT2,OT1]{fontenc}
\newcommand\cyr{%
\renewcommand\rmdefault{wncyr}%
\renewcommand\sfdefault{wncyss}%
\renewcommand\encodingdefault{OT2}%
\normalfont
\selectfont}
\DeclareTextFontCommand{\textcyr}{\cyr}

\usepackage{amssymb,amsmath}

\DeclareFontFamily{OT1}{rsfs}{}
\DeclareFontShape{OT1}{rsfs}{n}{it}{<-> rsfs10}{}
\DeclareMathAlphabet{\mathscr}{OT1}{rsfs}{n}{it}

\numberwithin{equation}{section}
\newtheorem{theorem}{Theorem}[section]
\newtheorem{lemma}[theorem]{Lemma}

\newtheorem{corollary}[theorem]{Corollary}
\newtheorem{claim}[theorem]{Claim}
\newtheorem{question}{Question}

\theoremstyle{definition}
\newtheorem{definition}[theorem]{Definition}

\theoremstyle{remark}

\newtheorem{example}[theorem]{Example}

\newcommand{\Ass}{\operatorname{Ass}}
\newcommand{\Assh}{\operatorname{Assh}}

\newcommand{\Ext}{\operatorname{Ext}}

\newcommand{\Supp}{\operatorname{Supp}}

\newcommand{\Hom}{\operatorname{Hom}}

\newcommand{\Ann}{\operatorname{Ann}}

\newcommand{\depth}{\operatorname{depth}}

\newcommand{\fq}{\frak{q}}

\newcommand{\calD}{\mathcal{D}}

\newcommand{\fka}{\mathfrak{a}}

\newcommand{\fkm}{\mathfrak{m}}
\newcommand{\fkn}{\mathfrak{n}}

\newcommand{\fkp}{\mathfrak{p}}
\newcommand{\fkq}{\mathfrak{q}}

\newcommand{\fkt}{\mathfrak{t}}
\newcommand{\fku}{\mathfrak{u}}

\newcommand{\q}{\mathfrak{q}}

\newcommand{\m}{\mathfrak{m}}

\def\H{\operatorname{H}}

\def\depth{\operatorname{depth}}
\def\Supp{\operatorname{Supp}}
\def\dim{\operatorname{dim}}
\def\Ann{\operatorname{Ann}}
\def\Ass{\operatorname{Ass}}
\def\Assh{\operatorname{Assh}}

\def\e{\operatorname{e}}

\begin{document}
\title{A note on Chern coefficients and  Cohen-Macaulay rings}

\author [N.T.T. Tam]{Nguyen Thi Thanh Tam}
\address{Hung Vuong University, Viet Tri, Phu Tho, Vietnam}
\email{thanhtamnguyenhv@gmail.com }

\author[H.L. Truong]{Hoang Le Truong}
\address{	Mathematik und Informatik, Universit\"{a}t des Saarlandes, Campus E2 4,
	D-66123 Saarbr\"{u}cken, Germany}
\address{Institute of Mathematics, VAST, 18 Hoang Quoc Viet Road, 10307
Hanoi, Viet Nam}
\address{Thang Long Institute of Mathematics and Applied Sciences, Hanoi, Vietnam}
\email{hoang@math.uni-sb.de\\hltruong@math.ac.vn\\
	truonghoangle@gmail.com}
	
\thanks{
This work is partially supported by a fund of Vietnam National Foundation for Science
and Technology Development (NAFOSTED) under grant number
101.04-2017.14. }



\begin{abstract} 
		In this paper, we investigate the relationship between the index of reducibility and Chern coefficients for primary ideals. As an application, we give characterizations of a Cohen-Macaulay ring in terms of its type, irreducible multiplicity, and Chern coefficients with respect to certain parameter ideals in Noetherian local rings.
\end{abstract}

\maketitle

\section{Introduction}
Throughout this paper, let $(R,\m)$ be a homomorphic image of a Cohen-Macaulay local ring with the infinite residue field $k$, $\dim R=d>0$, and $M$ a finitely generated $R$-module of dimension $s$. A submodule $N$ of $M$ is called an {\it irreducible submodule} if $N$ can not be written as an intersection of two properly larger submodules of $M$. The number of irreducible components of an irredundant irreducible decomposition of $N$, which is independent of the choice of the decomposition  by E. Noether \cite{N21}, is called the {\it index of reducibility} of $N$ and denoted by $\mathrm{ir}_M(N)$.  For an $\m$-primary ideal $I$ of $M$, {\it the index of reducibility} of $I$
on $M$ is the index of reducibility of $I M$, and denoted by $\mathrm{ir}_M(I)$. Moreover, we have $\mathrm{ir}_M(I) = \dim_k \mathrm{Soc}(M/I M)$, where we denote by $\mathrm{Soc}(N)$ the dimension of the socle of an $R$-module $N$ as a $k$-vector space.
 In the case $I$ is a parameter ideal of $M$, several properties of $\mathrm{ir}_M(I)$ had been found and played essential roles in the earlier stage of development of the theory of Gorenstein rings and/or Cohen-Macaulay rings. Recently, the index of reducibility of parameter ideals has been used to deduce a lot of information on the structure of some classes of modules, such as regular local rings by W. Gr\"{o}bner \cite{G51}; Gorenstein rings by Northcott, Rees \cite{N21,N57,Tr14,Tr19}; Cohen-Macaulay modules by D.G. Northcott, N.T. Cuong, P.H. Quy \cite{CQT15,Tr14,Tr15,Tr19}; Buchsbaum modules by S. Goto, N. Suzuki and H. Sakurai \cite{GS03,GS84}; generalized Cohen-Macaulay modules by N.T. Cuong, P.H. Quy and the second author \cite{CQ11, CT08} (see also \cite{Tr13, Tr19} for other modules). The aim of our paper is to continue this research direction. Concretely, we will give characterizations of a Cohen-Macaulay ring in terms of its {\it type} and Chern coefficients with respect to $g$-parameter ideals (See Definition \ref{gsp}).
Recall that {\it type} of a module $M$ was first introduced  by S. Goto and N. Suzuki \cite{GS84}, and is defined  as the supremum 
$$\mathrm{r}(M)=\sup_{\fkq}\mathrm{ir}_M(\fkq),$$ where $\fkq$ runs through the parameter ideals of $M$.  
It well-known  that there are integers $\e_i(I;M)$, called the {\it Hilbert coefficients} of $M$ with respect to $I$ such that for $n \gg 0$
\begin{eqnarray*}
 \ell_R(M/{I^{n+1}}M)={\e}_0(I;M) \binom{n+s}{s}-{\e}_1(I;M) \binom{n+s-1}{s-1}+\cdots+(-1)^s {\e}_s(I;M).
\end{eqnarray*}
Here $\ell_R(N)$ denotes, for an $R$-module $N,$ the length of $N.$ In particular, the leading coefficient $\e_0(I)$ is said to be {\it the multiplicity} of $M$ with respect to $I$ and $\e_1(I)$, which Vasconselos (\cite{V8}) refers to as the {\it Chern coefficient} of $M$ with respect to $I$. 
Now our motivation stems from the work of  \cite{V8}. Vasconcelos posed {\it the Vanishing Conjecture}:  $R$ is a Cohen–Macaulay local ring if and only if $\e_1(\q, R) = 0$ for some parameter ideal $\q$ of $R$.
It is shown that the relation between Cohen-Macaulayness and the Chern number of parameter ideals is quite surprising.  In \cite{Tr14}, motivated by some deep results of \cite{CGT,GNi} and also by the fact that this is true for $R$ is unmixed as shown in \cite{GGHOPV}, it was asked whether  the characterization of the Cohen-Macaulayness in terms of the Chern number of non-parameter ideals and the type of $R$ in the case that $R$ is mixed. Concretely,  the goal of this note is to understand the nature of  the following open question.

\begin{question}\rm 
     Is  $R$ is Cohen-Macaulay if and only if there exists a parameter ideal $\q$ of $R$ such that
$$\mathrm{r}(R)\le \e_1(\q:\m)-\e_1(\q).$$


\end{question}

Our main result partially answers the question in the following way.

\medskip

\begin{theorem}\label{T6.6}
Assume that $d=\dim R\ge 2$. The following statements are equivalent.
\begin{enumerate}[{(i)}] \rm
    \item {\it $R$ is Cohen-Macaulay.}

 \item {\it For all parameter ideals $\frak q$, we
    have}
$$\mathrm{r}(R)\le \e_1(\q:\m)-\e_1(\q).$$

    \item {\it For some $g$-parameter ideal $\frak q\subseteq \m^2$, we
    have}
$$\mathrm{r}(R)\le \e_1(\q:\m)-\e_1(\q).$$
\end{enumerate}

\end{theorem}

N.T. Cuong et al. \cite{CQT15} showed that there are integers $f_i(I;M)$, called the {\it irreducible  coefficients} of $M$ with respect to $I$ such that for $n \gg 0$
$$\mathrm{ir}_M(I^{n+1})=\ell_R([I^{n+1}M :_M \fkm]/I^{n+1}M )=\sum\limits_{i=0}^{s-1}(-1)^if_i(I;M)\binom{n+s-1-i}{s-1-i}.$$
The leading coefficient $f_0(I;R)$ is  called the irreducible multiplicity of $I$.
From the notations given above, the second main result is stated as follows.

\medskip

\begin{theorem}\label{T0.1}
Let $M$ be a finitely generated $R$-module of dimension $s\ge 2$.   Then the following statements are equivalent.
\begin{enumerate}[{(i)}] \rm
		\item {\it $M$ is Cohen-Macaulay.}
		\item {\it For all parameter ideals $\frak q$ of $M$, we
		have
		$$\mathrm{r}(M)\le f_0(\frak q,M).$$}
		\item {\it For some $g$-parameter ideal $\fkq $ of $M$, we
		have
		$$\mathrm{r}(M)\le f_0(\frak q,M).$$
		}

\end{enumerate}
	
\end{theorem}

From this main result, we obtain the following results.

\medskip

\noindent

\begin{corollary}\label{P4.3}
	Assume that $R$ is non-Cohen-Macaulay local ring with $d \geq 2$. Then we have
	$$ e_1(\q:\fkm)-e_1(\fkq) \le f_0(\frak q;R) \le \mathrm{r}(R).$$
	for all  $g$-parameter ideals $\q \subseteq \fkm^2$. 
\end{corollary}

The remainder of this paper is organized as follows. In the next section, we prove some preliminary results on the irreducible multiplicity and index of reducibility for $g$-parameter ideals. In the last section, we prove the main results and their consequences.

\section{Notations and preliminaries}

Throughout this section, assume that $R$ is a Noetherian local ring with maximal ideal $\m$, $d= \dim R>0$ with the infinite residue field $k=R/\m$ and $I$ is an $\m$-primary ideal of $R$. Let $M$ be a finitely generated $R$-module of dimension $s$. We denote $H^i_{\m}(M)$ as the $i$-th local cohomology module of $M$ with respect to $\m$.  Set $r_i(M) = \dim_{R/\frak m}((0):_{H^i_{\frak m}(M)} \frak m)$. 
Recall that a submodule $N$ of $M$ is irreducible if $N$ can not be written as the intersection of two properly larger submodules of $M$. Every submodule $N$ of $M$ can be expressed as an irredundant intersection of irreducible submodules of $M$. The number of irreducible submodules appearing in such an expression depends only on $N$,  but not on the expression. That number is called the index of reducibility of $N$ and is denoted by $\mathrm{ir}_M(N)$. In particular, if $N = IM$, then we have 
 $$\mathrm{ir}_M(I) : = \mathrm{ir}_M(IM) =\ell_R([I M :_M \fkm]/I M).$$

 It is well known that  by Lemma 4.2 in \cite{CQT15}, there exists a polynomial $p_{I,M}(n)$ of degree $s-1$ with rational coefficients such that  $$p_{I,M}(n)=\mathrm{ir}_M(I^{n+1})=\ell_R([I^{n+1} M :_M \fkm]/I^{n+1} M)$$  for all large enough $n$. Then, there are integers $f_i(I;M)$ such that
$$p_{I,M}(n)=\sum\limits_{i=0}^{s-1}(-1)^if_i(I;M)\binom{n+s-1-i}{s-1-i}.$$
These integers  $f_i(I;M)$   are called the irreducible coefficients of $M$ with respect to $I$. In particular, the leading coefficient  $f_0(I; M)$ is  called the irreducible multiplicity of $M$ with respect to $I$.
The readers may refer to \cite{Tr14,Tr15,CQT15} for more characterizations of the Cohen-Macaulayness of $M$ in terms of the coefficient $f_0(\fq, M)$. In \cite{CQT15}, the authors studied the function $\mathrm{ir}_M(\fq^{n+1})$ when $M$ is generalized Cohen-Macaulay and $\fq$ is a standard parameter ideal of $M$. Recall that an  $R$-module $M$ is said to be a {\it generalized Cohen-Macaulay module} if 
$H^i_\fkm(M)$ are of finite length for all $i=0,1,\ldots, s-1$ (see \cite {CST}). This condition is equivalent to say that there exists a  parameter ideal $\fkq=(x_1,\ldots,x_s)$ of $M$ such that $\fkq H^i_\fkm(M/\fkq_j M)=0$ for all $0\le i+j < s$, where $\q_j=(x_1,\ldots, x_j)$ (see \cite {T}), and such a parameter ideal is called a {\it standard parameter ideal} of $M$. It is well-known that if $M$ is a generalized Cohen-Macaulay module,  every parameter ideal of $M$ in a large enough power of the maximal ideal $\fkm$ is standard (see \cite{T83,T}). 

For proving the main result in next section, we need the following auxiliary lemma, which is  shown in \cite[Lemma 2.1]{Tr14}. 

%
%
%

\begin{lemma}\label{L000}
Assume that $N$ is a submodule of $M$ such that $\dim N<\dim M$. Then we have
$$f_0(I;M)\leqslant f_0(I;M/N).$$
\end{lemma}

The following  lemma shows the existence of a special superficial  element which is useful in many inductive proofs in the sequel.

\begin{lemma}\label{cut}
Let $M$ be a finitely generated $R$-module of dimension $s>1$ and $I$ an $\fkm$-primary ideal of $R$.
Assume that $x$ is a superficial element of $M$ with respect to $I$ such that $H^0_\fkm(M)=(0):_Mx$. Then we have
$$f_0(I;M)\leqslant f_0(I;M/xM).$$
\end{lemma}
\begin{proof}
Let $W=H^0_\fkm(M)$. Since $x$ is a superficial element of $M$ with respect to $I$, we have 
$I^{n+1}M:_Mx=I^nM+(0):_Mx=I^nM+W$
for large enough $n$.  Therefore it follows from that the sequences 
$$0\to M/ (I^nM+W)\to M/ I^{n+1}M\to M/(x, I^{n+1})M\to 0$$
are exact for large enough $n$, that we get the following exact sequence
$$0\to ((I^nM+W):_M\frak m)/(I^nM+W)\to (I^{n+1}M:_M\frak m)/I^{n+1}M\to ((x,I^{n+1})M:_M\frak m)/(x,I^{n+1})M.$$
\begin{claim}
$(I^nM + W):_M \frak  m = (I^nM :_M \frak m )+ W$ for large enough $n$.  
\end{claim}

\begin{proof}
Since $W$ has finite length, there exists an integer $n_0$ such that
$\frak m^{n_0}M \cap W=(0)$. 
There exists a superficial element $x_1$ of $\overline M=M/W$ with respect to $I$. Therefore, there exists an integer $n_1$ such that $$(I^nM + W):_M \frak  m \subseteq (I^nM + W):_M x_1 = I^{n-1}M + W,$$
for all $n>n_1$. Let $n$ be an integer such that $n>\max\{n_0,n_1\}+1$. Let $a \in (I^nM + W):_M \frak  m$, we have $a = b + c$ with $b \in I^{n-1}M$
and $c \in W$. Hence $\frak m b \in I^{n-1}M
\cap (I^nM + W) = I^nM$. So $b \in I^nM :_M \frak  m$. Then
$$(I^nM + W):_M \frak  m = (I^nM :_M \frak m )+ W$$
for all $n>\max\{n_0,n_1\}+1$.
\end{proof}
We have 
$$\begin{aligned}
&\ell_R(((x,I^{n+1})M:_M\frak m)/(x,I^{n+1})M)\\
&\ge \ell_R((I^{n+1}M:_M\fkm)/I^{n+1}M)-\ell_R(((I^{n}M+W):_M\fkm)/(I^{n}M+W))\\
&=\ell_R((I^{n+1}M:_M\fkm)/I^{n+1}M)-\ell_R((I^{n}M:_M\fkm+W)/(I^{n}M+W))\\
&=\ell_R((I^{n+1}M:_M\fkm)/I^{n+1}M)-\ell_R((I^{n}M:_M\fkm)/I^{n}M)+\ell_R((I^nM:_M\fkm)\cap W).
\end{aligned}
$$
Since $W$ has finite length and $s>1$,  we have 
$$f_0(I;M)\leqslant f_0(I;M/xM),$$
as required. 
\end{proof}

In \cite{Tr19} the second author proved important properties of $g$-parameter ideals which we will need to prove our main result. In the last part of this section we recall part of these results. We refer to \cite{Tr19} and to \cite{Tr14} for details. Put $\Assh_RM=\{\fkp \in \Supp_RM \mid \dim R/\fkp=s\},$
where $\Supp_RM$ is the support of $M$. Compared with the set of associated primes, we have $\Assh_RM \subseteq  \Ass_RM.$
Let $\Lambda (M)=\{\dim_R N \mid N  \ \text{is an}\ R\text{-submodule of}\ M, N \ne (0)\}.$ Then we have $$\Lambda (M)=\{\dim R/\fkp \mid \fkp \in\Ass_RM\}.$$ 
We put  $\fkt=\sharp \Lambda (M),$ and   $$ \Lambda(M) =\{ 0 \le d_1<d_2<\cdots<d_{\fkt}=s \}.$$
Because $R$ is Noetherian, $M$ contains the largest submodule $D_i$ with $\dim_RD_i=d_i$, for all $1\leq i\leq \fkt$. Then the filtration 
$$\calD : D_0=(0)\subsetneq D_1\subsetneq D_2\subsetneq \cdots 
\subsetneq D_{\fkt}=M$$
of submodules of $M$ is called {\it the dimension filtration} of $M$. 
The notion of dimension filtration was first given by P. Schenzel \cite{SCH}. 
Our notion of dimension filtration is different from that of \cite{CC, SCH},
however throughout  this paper let us unite the above definition. Notice that, if $(0)=\bigcap_{\fkp \in \Ass_RM}N(\fkp)$ is a reduced primary decomposition of the submodule $(0)$  of $M$,  then  $D_i=\bigcap\limits_{\fkp \in \Ass_RM,\ \dim R/\fkp \geq d_{i+1}}N(\fkp)$ for all $1\le i\le \fkt-1$, and so $D_{\fkt -1}$ 
is also called the {\it unmixed component} of $M$.
 For all $1 \le i \le \fkt$, we put $C_i = D_i/D_{i-1}$.  Then $\Ass_RC_i=\{ \fkp \in \Ass_RM \mid \dim R/\fkp =d_i\}$, and $\Ass_RM/D_i=\{\fkp \in \Ass_RM \mid \dim R/\fkp 	\geq d_{i+1} \}$.\\

 A  system $x_1,x_2, \ldots, x_s$  of parameters of $M$ is called {\it distinguished}, if  $$  (x_j \mid d_{i} < j \le s) D_i=(0),$$ for all $1\le i\le \fkt$. A parameter ideal $\fkq$ of $M$ is called {\it distinguished}, if there exists a distinguished system $x_1,x_2, \ldots, x_s$ of parameters of $M$ such that $\fkq=(x_1,x_2, \ldots, x_s)$. Notice that, if $M$ is a Cohen-Macaulay $R$-module, every parameter ideal of $M$ is distinguished. It is well-known  that distinguished systems of parameters always exist (see \cite{SCH}). If $x_1, x_2,\ldots, x_s$ is a distinguished system of parameters of $M$, so are $x_1^{n_1},x_2^{n_2}, \ldots, x_s^{n_s}$ for all integers $n_j \ge 1$.

 We denote by $\fkq_i$  the ideal $(x_1,\ldots,x_i)$ for $i=1,\ldots , d$  and stipulate that  $\fkq_0$ is  the zero ideal of $R$. Then the sequence $x_1, x_2, \ldots , x_m\in\m$  is called a $d$-sequence on $M$ if
$$\fkq_iM :_M x_{i+1}x_j = \fkq_iM :_M x_j$$
for all $0 \le i < j \le m$. The concept of a $d$-sequence is given by Huneke \cite{H82} and it plays an important role in the theory of Blow up algebras, i.e. Rees algebras.
We now present the main object of this paper.

\begin{definition}[{cf. \cite[Definition 2.3]{Tr19}}]\label{gsp}\rm
	A distinguished system $x_1, x_2,\ldots, x_s$ of parameters of $M$ is called a {\it g-system of parameters} on $M$, if it is a $d$-sequence, and we have
		 $$\Ass(C_i/\q_j C_i)\subseteq \Assh(C_i/\q_j C_i)\cup \{\fkm\},$$
 for all  $0 \le j \le s -1$ and $0 \le i \le \fkt$.	 A parameter ideal $\fkq$ of $M$ is called {\it $g$-parameter ideal}, if there exists a $g$-system $x_1,x_2, \ldots, x_s$ of parameters of $M$ such that $\fkq=(x_1,x_2, \ldots, x_s)$.
\end{definition}
The readers can refer some facts about $g$-systems of parameters in \cite{Tr19}. Notice that  $g$-system of parameters always exist.  Moreover if $M$ is a generalized Cohen-Macaulay module, then by the definition of $g$-parameter ideal, every $g$-parameter ideal of $M$ is  standard. Besides,
we have following property.

\begin{lemma} \label{L2.2} Let $x_1, x_2,\ldots, x_s$ form a $g$-system of parameters on $M$ with $s\ge 2$. Let $0\le j\le s-2$ and $N/\fkq_{j}M$ denote the unmixed component of $M/\fkq_{j}M$. Assume that $M/N$ is Cohen-Macaulay. Then $C_\fkt$ is Cohen-Macaulay.

\end{lemma}

\begin{proof}
	
	We may assume that $s \ge 3$ and $j=1$. 	For a submodule $L$ of $M$, we denote $\overline L=(L+xM)/xM$ with $x:=x_1$. By the definition of $g$-system of parameters, we have  $\Ass(C_\fkt/xC_\fkt)\subseteq \Assh(C_\fkt/xC_\fkt)\cup\{\fkm\}$. Therefore $\overline N/\overline D_{\fkt-1}$ has  finite length. Since $M/N$ is a Cohen-Macaulay module, $\H^i_\fkm(M/D_{\fkt-1}+xM)=0$ for all $0<i<s-1$. Therefore, we derive from the exact sequence
	$$0 \to C_\fkt\overset{.x}\to C_\fkt \to  M/D_{\fkt-1}+xM \to 0$$
	the following exact sequence:
	$$0 \to \H^0_\m(M/D_{\fkt-1} + xM) \to \H^1_\m(C_\fkt)
	\overset{.x}\to \H^1_\m(C_\fkt) \to 0.$$
	Thus  $\H^1_\fkm(C_\fkt)=0$,  so $\overline N/\overline D_{\fkt-1}=\H^0_\fkm(M/D_{\fkt-1}+xM)=0$. Hence $N= D_{\fkt-1}+xM$. Moreover, since $x$ is  $C_\fkt$-regular and  $C_\fkt/xC_\fkt\cong \overline M/\overline D_{\fkt-1} =  M/N$ is a Cohen-Macaulay module, $C_\fkt$  is Cohen-Macaulay.
\end{proof}

\section{Proof of Main Theorems and Corollaries}
In this section we prove the main results of this paper. Recall that $$\mathrm{r}(M)=\sup\{\mathrm{ir}_M(\q)\mid \q \textrm{ is a parameter ideal for }M\},$$
and it is called {\it the type} of $M$. In particular,  if $M=R$, we simply denote  by $\mathrm{r}(R)$ the type of the local ring  
$R$ as a module over itself. The notion of the type of a module $M$ was first introduced by Goto and Suzuki (\cite{GS84}). With notation as above we have the following lemma.

\begin{lemma}\label{F2.5}
	Assume that $M/N$ is Cohen-Macaulay with $\dim N < \dim M$. Then there exists a parameter ideal $Q$ such that  $$ \mathrm{ir}_{N}(Q)+ r_s(M)\le \mathrm{r}(M).$$

\end{lemma}
\begin{proof}
By Lemma 3.6 in \cite{GS03} there exists a parameter ideal $Q$ of $M$ such that  the canonical map
	$$\xymatrix{\phi_{M}:M/QM\ar[r]& \H^{s}_\m(M)}$$
	is surjective on the socles. Put $L=M/N$. Then we look at the exact sequence
	$$\xymatrix{0\ar[r]& N\ar[r]^\iota& M \ar[r]^\epsilon & L \ar[r]&0}$$
	of $R$-modules, where $\iota$ (resp. $\epsilon$) denotes the canonical embedding (resp. the canonical epimorphism).
	Then, since $\dim N < \dim M$ and since  $L$ is Cohen-Macaulay, we get the following commutative diagram 
	$$\xymatrix{
		0\ar[r]& N/Q N\ar[r]^{\overline \iota}& M/Q M \ar[d]^{\phi_M}\ar[r]^{\overline \epsilon} & L/Q L \ar[r] \ar@{^{(}->}[d]^{\phi_{L}}&0\\
		& & \H^s_\m(M) \ar[r]^{=} & \H^s_\m(L) &}$$
with  exact first row on socles. In fact,
	let $x\in (0):_{L/QL}\m$. Then, since $\phi_M$ is surjective on the socles, we get  an element $y\in (0):_{M/\q M}\m$ such that $\phi_{L}(x)=\phi_M(y)$. Thus $\overline{\epsilon}(y)=x$, because the canonical map $\phi_{L}$ is injective. 
	Therefore,
we
have 	$$\mathrm{ir}_M(Q)=\mathrm{ir}_{N}(Q)+\mathrm{ir}_{M/N}(Q)=\mathrm{ir}_{N}(Q)+ r_s(M).$$
Thus by the definition of type, we have
 $$ \mathrm{ir}_{N}(Q)+ r_s(M)\le \mathrm{r}(M).$$

\end{proof}

Recall that $M$ is {\it sequentially Cohen-Macaulay}, 
if  $C_i=D_i/D_{i-1}$   is Cohen-Macaulay for all $1 \le i \le \fkt$.
Notice that every module of dimension $1$ is sequentially Cohen-Macaulay. 
\begin{corollary}
Assume that $M$ is sequentially Cohen-Macaulay of dimension $s\ge 2$. Then we have
$$\sum\limits_{i\in\Bbb Z}r_i(M)\le \mathrm{r}(M).$$

\end{corollary}
\begin{proof}
It follows from Lemma \ref{F2.5} and the definition of sequentially Cohen-Macaulay modules.
\end{proof}

\begin{lemma}\label{P4.60}
Let $M$ be a finitely generated $R$-module of dimension $s\ge 2$ and $\fkq$  a $g$-parameter  ideal of $M$ such that
 $$\mathrm{r}(M)\le f_0(\fkq,M).$$
	Then $C_\fkt$ is Cohen-Macaulay.
\end{lemma}

\begin{proof}
		Assume that $\q$ is generated by the $g$-system $x_1,\ldots,x_{s}$ of parameters of $M$. 
	Put   $A=M/\q_{s-2}M,$ and let $N$ denote the unmixed component of $A$. 
Korollar 2.2.4 in \cite{SCH82} say that if $R$ is a homomorphic image of
a Cohen-Macaulay local ring with the infinite residue field $k$ and
$M$ is a finitely generated $R$-module
then $\dim R/\fka_i(M)\le i$,
where $\fka_i(M)=\Ann H^i_\fkm(M)$ for all $i$. Moreover, since $\dim A=2$ and $\dim R/\fka_1(A)\le 1$,  we can choose $z$ such that $z$ is a parameter element of $A$
and $z\in \fka_1(A)$. Therefore
$zH^1_\fkm(A)=0$. 
Now, we derive from that $ N $ is the unmixed component of $A$  and the exact sequence
	$$0 \to N\to A \to A/N \to 0$$
	 that the map
$ H^1_\fkm(N )\to H^1_\fkm(A)$ is injective.
Therefore $zH^1_\fkm(N)=0$,  so $z \in \Ann N$. Thus, $x_{s-1},z$ is a distinguished system of parameters for $A$.
Since $\Ann D_{\fkt-1}+\q_{s-1}$ is an $\m$-primary ideal of $R$, we can choose $y:=z^m \in
  \Ann(D_{\fkt-1})$ for large enough $m$  such that $yH_\m^1 (A) = 0$ and $(0):_Ay=(0):_Ay^2$.

%

%
	Since $x_{s-1},y$ is a distinguished system of parameters for $A$, by Lemma 2.3 in \cite{Tr13}, we have $N=(0):_Ay^n$ for all $n$.  Hence we have that the following commutative diagram with exact rows for all $n\ge 2$.
	$$
	\xymatrix{
		0 \ar[r] & A/N \ar[r]^{\cdot y} \ar[d]^{\mathrm{id}} & A \ar[r] \ar[d]^{\cdot y^n} & A/yA \ar[r] \ar[d] & 0\\
		0 \ar[r] & A/N \ar[r]^{\cdot y^{n+1}} & A \ar[r] & A/y^{n+1}A \ar[r]  & 0.
	}
	$$
	We apply the functor $\H^1_\m(\bullet)$ to the above diagram to get the commutative diagram
	$$
	\xymatrix{
		 \H^1_\m(A/N)\ar[r]^{\alpha_1} \ar[d]^{\mathrm{id}} & \H^1_\m(A) \ar[d]^{\cdot y^n}\\
		 \H^1_\m(A/N)\ar[r]^{\alpha_{n+1}} & \H^1_\m(A),
	}
	$$ 
	where $\alpha_1$, $\alpha_{n+1}$ are canonical homomorphisms. Thus, $\alpha_{n+1}=y^n \circ \alpha_1=0$ for all $n\ge 3$, because of  the choice of $y$. Therefore  we get the commutative diagram
with exact rows $$
\xymatrix{
	0 \ar[r] & \H^0_\m(A) \ar[r] \ar[d]^{y} & \H^0_\m(A/y^nA) \ar[r] \ar[d] & \H^1_\m(A/N) \ar[r] \ar[d]^{\mathrm{id}} & 0\\
	0 \ar[r] & \H^0_\m(A) \ar[r]  & \H^0_\m(A/y^{n+1}A) \ar[r] & \H^1_\m(A/N) \ar[r]  & 0.
}
$$
 By applying the functor $\Hom(k;\bullet)$ to this diagram, we obtain a commutative diagram
$$
\xymatrix{
	\Hom(k,\H^1_\m(A/N))\ar[r]^{\beta_n} \ar[d]^{\mathrm{id}} & \Ext^1(k,\H^0_\m(A)) \ar[d]^{\cdot y}\\
		\Hom(k,\H^1_\m(A/N))\ar[r]^{\beta_{n+1}}  & \Ext^1(k,\H^0_\m(A)),
}
$$ 
	where $\beta_n$, $\beta_{n+1}$ are connecting homomorphisms. Thus, $\beta_{n+1}=y \circ \beta_n=0$, since $y\in \m$. Therefore the sequence
	$$0\to (0):_{\H^0_\m(A)}\m\to (0):_{\H^0_\m(A/y^nA)}\m\to (0):_{\H^1_\m(A/N)}\m\to 0$$
	 is exact,  so $r_0(A/y^nA)= r_0(A)+r_1(A/N)$ for all $n\ge 3$. 
	 
	 On the other hand, since $\alpha_{n+1}=0$ for all $n\ge 1$, we get the commutative diagram
	$$
	\xymatrix{
		0 \ar[r] & \H^1_\m(A) \ar[r] \ar[d]^{y} & \H^1_\m(A/y^nA) \ar[r] \ar[d] & (0):_{\H^2_\m(A/N)}y^n \ar[r] \ar[d]^{\mathrm{id}} & 0\\
		0 \ar[r] & \H^1_\m(A) \ar[r]  & \H^1_\m(A/y^{n+1}A) \ar[r] & (0):_{\H^2_\m(A/N)}y^{n+1} \ar[r]  & 0,
	}
	$$ 
	 whose rows are exact sequences. By applying the functor $\Hom(k;\bullet)$ to this diagram, we obtain a commutative diagram
	 $$
	 \xymatrix{
	 	\Hom(k,(0):_{\H^2_\m(A/N)}y^n)\ar[r]^{\gamma_n} \ar[d]^{\mathrm{id}} & \Ext^1(k,\H^1_\m(A)) \ar[d]^{\cdot y}\\
	 	\Hom(k,(0):_{\H^2_\m(A/N)}y^n)\ar[r]^{\gamma_{n+1}}  & \Ext^1(k,\H^1_\m(A)),
	 }
	 $$ 
	 where $\gamma_n$, $\gamma_{n+1}$ are connecting homomorphisms. Thus, $\gamma_{n+1}=y \circ \gamma_n=0$, since $y\in \m$. Therefore the sequence
	 $$0\to (0):_{\H^1_\m(A)}\m\to (0):_{\H^1_\m(A/y^nA)}\m\to (0):_{\H^2_\m(A/N)}\m\to 0$$
	 is exact. Since $\dim N<2$, so we have 
	 $r_1(A/y^nA)= r_1(A)+r_2(A/N)=r_1(A)+r_2(A)$
	 for all $n\ge 3$. 
	 Since $\dim A/y^nA=1$, 	by Lemma \ref{F2.5}, we have
	 	 $$r_0(A/y^nA)+r_1(A/y^nA)\le \mathrm{r}(A/y^nA).$$
	By the definition of type of local rings and the hypothesis,  we have 
	$$\begin{aligned}
	r_0(A)+r_1(A/N)+r_1(A)+r_2(A)&=r_0(A/y^nA)+r_1(A/y^nA)\\
	&\le \mathrm{r}(A/y^nA)\le \mathrm{r}(M)
	\le f_0(\fkq,M).
	\end{aligned}$$ 
On the other hand, by Lemma \ref{cut} and Lemma \ref{L000}, we have
	$$\begin{aligned}
	 f_0(\fkq,M)\le f_0(\fkq,A)\le f_0(\fkq,A/N).  
	\end{aligned}$$ 
	Since $\dim A/N=2$ and $N$ is the unmixed component of $A$, $A$ is generalized Cohen-Macaulay. Thus we can choose a superficial element $y_0$ of $A/N$ with respect to $\q$ such that $y_0\H^1_\m(A/N)=0$. It follows from the  exact sequence 
	$$\xymatrix{0\ar[r] & A/N \ar[r]^{.y_0} &A/N \ar[r]&A/(y_0A+N)\ar[r]&0}$$
	and $y_0H^1_\m(A/N)=0$
	that  the sequence
$$\xymatrix{0\ar[r] & \H^1_\m(A/N) \ar[r] & \H^1_\m(A/(y_0A+N))\ar[r]&\H^2_\m(A/N) \ar[r]&0}$$
is exact. By applying the functor $\Hom(k;\bullet)$, we have 
$$r_1(A/(y_0A+N))\le r_1(A/N)+r_2(A/N)$$  
Since  $y_0$ is $A/N$-regular, by Lemma \ref{cut}, we have 	
 $$f_0(\frak q;A/N)\le f_0(\frak q;A/(y_0A+N)).$$
Because $\dim (A/(y_0A+N))=1$,  we have $f_0(\frak q;A/(y_0A+N))\le r_1(A/(y_0A+N))$ and then
 $$f_0(\frak q;A/N)\le r_1(A/N)+r_2(A/N).$$
       Thus, $r_1(A)=0$,  so  $\H^1_\fkm(A)=0$. 
It follows from the exact sequence
$$0\to N\to A\to A/N\to 0$$
that the sequence
$$0\to \H^1_\fkm(N)\to \H^1_\fkm(A) \to  \H^1_\fkm(A/N)\to 0 $$
is exact and $\H^1_\fkm(A/N)=0$. Hence, $C_{\fkt}$ is Cohen-Macaulay, because of Lemma \ref{L2.2}.
\end{proof}

We are now ready to prove the main theorems of this section.

\begin{proof}[{\bf Proof of Theorem \ref{T0.1}}]

$(i)\Rightarrow (ii)$ follows from Theorem 5.2 in \cite{CQT15}.\\
$(ii) \Rightarrow (iii)$ is trivial.\\
$(ii) \Rightarrow (i)$  By Lemma \ref{P4.60}, we have $C_{\fkt}=M/D_{\fkt-1}$ is Cohen-Macaulay. 
By Lemma \ref{F2.5}, there exists a parameter ideal $Q$ such that 
$$ \mathrm{ir}_{D_{\fkt-1}}(Q)+ r_s(M)\le \mathrm{r}(M)$$
 It follows from Lemma \ref{L000} and that $M/D_{\fkt-1}$ is Cohen-Macaulay, that we have
 $$\mathrm{r}(M)\le f_0(\q;M)\le f_0(\q;M/D_{\fkt-1})=r_s(M).
 $$
Therefore, we have
$\mathrm{ir}_{D_{\fkt-1}}(Q)=0$, so $D_{\fkt-1}=0$. Hence, $M$ is Cohen-Macaulay, as required.

\end{proof}

\noindent
{\bf Proof of Theorem \ref{T6.6}.} $(i)\Rightarrow (ii)$ follows from Theorem 1.1 in \cite{Tr15}.\\
$(ii) \Rightarrow (iii)$ is trivial.\\
$(iii) \Rightarrow (i)$ 	Let $I=\q:\m$. First, assume that $e_0(\m;R)>1$. Then by Proposition 2.3 in \cite{GS03}, we get that $\m I^n = \m\q^n$ for all $n$. Thus, $I^n\subseteq \q^n:\m$ for all $n$. It follows that
	$$\ell_R(R/\fkq^{n+1})-\ell_R(R/I^{n+1})=\ell_R(I^{n+1}/\fkq^{n+1})\le \ell_R((\q^{n+1}:\m)/\q^{n+1}).$$ 
	Therefore, $e_1(I;R)-e_1(\q;R)\le f_0(\q;R)$,  so $\mathrm{r}(M)\le f_0(\q;R)$. By Theorem \ref{T0.1}, $R$ is Cohen-Macaulay, as required.

	Now  assume that $e_0(\m;R)=1$. Let $\fku$ denote the unmixed component of $R$, and put $S=R/\fku$. Since $\dim \fku<\dim R$  we have $e_0(\m;S)=1$. By Theorem 40.6 in \cite{Na}, $S$ is Cohen-Macaulay,  so $S$ is a regular local ring. 
	By Lemma \ref{F2.5}, there exists a parameter ideal $Q$ of $R$ such that 
$$ \mathrm{ir}_{\fku}(Q)+ r_d(R)\le \mathrm{r}(R).$$
It follows from $\mathrm{r}(R)\le e_1(\q:\m)-e_1(\q)$ and the following claim that  $\mathrm{ir}_{\fku}(Q)=0$,  so $\fku=0$. Hence $R$ is Cohen-Macaulay, as required.

\begin{claim}
$e_1(\q:\m)-e_1(\q)\le r_d(R).$

\end{claim}
\begin{proof}
By Theorem 1.1 in \cite{Tr15}, we have $$e_1(\q S:\m S;S)-e_1(\q S;S)=r_d(S)=r_d(R).$$
Since $(\q:\m)S\subset \q S:\m S$, we have
$ \ell_R(S/(\q:\m)^{n+1}S)\ge \ell_R(S/(\q S:\m S)^{n+1} )$, for all $n\ge 0$. Since $S$ is a regular local ring and $\q\subseteq \m^2$, by Theorem 2.1 in \cite{CPV}, we have $\q S:\m S$ is integral over $\q S$ and so $e_0(\q S:\m S;S)=e_0((\q :\m)S;S)=e_0(\q S;S)$. Therefore, we have $$e_1((\q :\m )S;S)\le e_1(\q S :\m S;S).$$ 
Hence 
$$e_1((\q :\m )S;S)-e_1(\q S;S)\le r_d(R).$$	
On the other hand,	 by Lemma 3.4 in \cite{CGT}, we have 
\begin{center}
$e_1(\q)= \begin{cases}
e_1(\q,S) & \text{ if } \dim \fku\le d-2\\
e_1(\q,S)-e_0(\q,\fku) & \text{ if } \dim \fku= d-1
\end{cases}$
\quad $e_1(\q:\m)=\begin{cases}
e_1(\q:\fkm,S) & \text{ if } \dim \fku\le d-2\\
e_1(\q:\fkm,S)-e_0(\q:\m,\fku) & \text{ if } \dim \fku= d-1.
\end{cases}$
\end{center}

If $\dim \fku<d-1$, then we have  $e_1(\q:\m)-e_1(\q)=e_1(\q:\m,S)-e_1(\q,S)\le r_d(R),$ as required.

	Now, we can  assume that $\dim \fku= d-1$. Then we have
	$$e_1(\q:\m)-e_1(\q)\le r_d(R)-e_0(\q:\m,\fku)+e_0(\q,\fku).$$
	By the Prime Avoidance Theorem, we can choose a parameter element $x\in\q$ of $R$ such that $(x)\cap \fku=0$. Put $\bar R=R/(x)$, and $\bar S=S/(x)$. Then we have the following exact sequence
	$$0\to \fku \to \bar R\to \bar S\to 0.$$
	 Consequently, 
	 $$e_0(\fkm,\bar R)=e_0(\fkm,\bar S)+e_0(\fkm,\fku)\ge 2$$
	 so that by Proposition 2.3 in \cite{GS03}, $\fkq\bar R:\fkm=(\fkq:\fkm)\bar R$ is integral over $\q\bar R$. Thus, $e_0(\fkq:\fkm;\bar R)=e_0(\fkq,\bar R)$.  Since $x$ is regular in $S$, $e_0(\fkq:\fkm;\bar S)=e_0(\fkq,\bar S)$. Hence, we have
	 $$\begin{aligned}
	 e_1(\q:\m)-e_1(\q)&\le r_d(R)-e_0(\q:\m,\fku)+e_0(\q,\fku)\\&=r_d(R)-(e_0(\q:\fkm,\bar R)-e_0(\q:\m,\bar S))+(e_0(\fkq;\bar R)-e_0(\q,\bar S))=r_d(R).
	 \end{aligned}$$
\end{proof}


Let us note the following example of  non-Cohen-Macaulay local rings $R$  where we have
$f_0(\fkq;R) = e_1(\q:\m)-e_1(\q)$ for all parameter ideals $\q$. 
\vspace{3mm}

\begin{example}
 Let $d \geqslant 3$ be an integer and let $U = k[[X_1,X_2, \ldots ,X_d, Y ]]$ be the
formal power series ring over a field $k$. We look at the local ring
$R = U/[(X_1,X_2, \ldots ,X_d) \cap (Y )]$.
Then $R$ is a reduced ring with $\dim R = d$. We put $A = U/(Y )$ and $D =
U/(X_1,X_2,\ldots ,X_d)$. Let $\fkq$ be a parameter ideal in $R$. Then, since $D$ is a DVR and $A$ is a
regular local ring with $\dim A = d$, thanks to the exact sequence $0\to D \to R \to A \to 0$,
 we get that $\depth R = 1$ and  the sequence
$$\xymatrix{0\ar[r]&D/\fkq^{n+1}D\ar[r]&R/\fkq^{n+1}R\ar[r]&A/\fkq^{n+1}A\ar[r]&0}$$
is exact. By applying the functor
$\mathrm{Hom}_R(R/\frak m, \bullet)$ we obtain the following exact sequence
$$\xymatrix{0\ar[r]&[\fkq^{n+1}:_D\fkm]/\fkq^{n+1}\ar[r]&[\fkq^{n+1}:_R\fkm]/\fkq^{n+1}\ar[r]&[\fkq^{n+1}:_A\fkm]/\fkq^{n+1}\ar[r]&0}.$$
Therefore, we have
$$\begin{aligned}
\ell_R([\fkq^{n+1}:_R\fkm]/\fkq^{n+1})&=\ell_R([\fkq^{n+1}:_A\fkm]/\fkq^{n+1})+\ell_R([\fkq^{n+1}:_D\fkm]/\fkq^{n+1})\\
&=\binom{d-1+n-1}{d-1}+ 1
\end{aligned}$$
for all integers $n \geqslant 0$, whence $f_0(\fkq;R)=1=r_d(R)$ for every parameter ideal $\fkq$ in $\fkm$.
Since $A$ is regular local ring and $d\ge 3$, we have 
$$e_1(\q:\m)-e_1(\q)=r_d(R)$$ 
However $R$ is not a Cohen-Macaulay ring, since $H^1_\fkm(R) = H^1_\fkm(D)$ is not a finitely generated $R$-module,
where $\fkm$ denotes the maximal ideal in $R$, but $R$ is sequentially Cohen-Macaulay.

\end{example}

Now let us note the following example of  non-Cohen-Macaulay local rings $R$ of dimension $1$ where we have
$f_0(\fkq;R) =\mathrm{r}(R)$ for all parameter ideals $\q$. 

\begin{example}
Let $R=k[[X,Y,Z]]/((X^a)(X,Y,Z)+(Z^b))$ with $a$, $b\ge 2$, $A=k[[X,Y,Z]]$, $\fkn$ and $\fkm$ be the maximal ideal of $A$ and $R$, respectively. Then we have
\begin{center}
$\dim R=1 $, $\H^0_\fkm(R)=(X^a,Z^b)/((X^a)\fkn+(Z^b))\cong A/\fkn$ and $R/\H^0_\fkm(R)=A/(X^a,Z^b)$.
\end{center}
Therefore $R$ is not a Cohen-Macaulay ring. Now we claim that  for any parameter ideal $\q$ of $R$, we have
$$\mathrm{ir}_R(\q)=2.$$
Moreover $\mathrm{r}(R)\le f_0(\q;R)$ for any parameter ideal $\q$ of $R$.
Indeed, let $F=cX+dY+eZ$ be an element of $A$ having $\q$ is generated by its image in $R$. Then $F$, $X^a$, $Z^b$ form an $A$-regular, so we have an exact sequence
$$\xymatrix{0\ar[r]&A/\fkn\ar[r]& R/\fkq R\ar[r]&A/(X^a,Z^b,F)\ar[r]&0.}$$ 
Consider $\Delta=dX^{a-1}Z^{b-1}$, then $\Delta$ is a generator for both of the socles of $R/\fkq R$ and $A/(X^a,Z^b,F)$. Therefore $\mathrm{ir}_R(\q)=2,$ as required.
\end{example}


\end{document}